\documentclass{amsart}[12pt]

\usepackage{amsmath, endnotes, amsthm, enumerate, amsfonts, mathrsfs,xypic,pxfonts,comment,rotating}
\xyoption{all}

\newtheorem{definition}{Definition}
\newtheorem{theorem}{Theorem}

\newtheorem{lemma}{Lemma}

\title[Two-dimensional $C^{2,1}$ metric without $C^2$ embedding]{A two-dimensional $C^{2,1}$ metric with no local $C^2$ embedding in $\mathbb{R}^3$, following Pogorelov}
\author{Jonathan Holland}

\newcommand{\osarray}[2]{\overset{\displaystyle{\overset{\displaystyle{#1}}{#2}}}{\tfrac{}{}}}

\newcommand{\dynkinexample}{
$$
\def\objectstyle{\displaystyle}
\xymatrix{
&&&&&\bullet\\
\osarray{a}{\bullet}\ar@{-}[r]&\bullet\ar@{-}[r]&\bullet\ar@{..}[r]&\bullet\ar@{-}[r]&\bullet\ar@{-}[ur]\ar@{-}[dr]&\\
&&&&&\bullet
}
$$
}

\begin{document}
\maketitle 
\nocite{*}

\section{Introduction}
The purpose of this article is to examine a counterexample, due to A. V. Pogorelov \cite{Pogorelov}, of a two-dimensional Riemannian metric of class $C^{2,1}$ that is not locally realizable in $\mathbb{R}^3$ by a $C^2$ isometric embedding.  Pogorelov's argument proceeds as follows:
\begin{itemize}
\item First, construct for each radius $a$, a metric in the disc of radius that is not globally realizable in the disc and that interpolates (in a $C^{2,1}$ fashion) the flat metric at the boundary.
\item Then take a sequence of such metrics in disjoint balls of decreasing radius in $\mathbb{R}^2$ that accumulate at a single point.  The resulting metric will not be $C^2$ embeddable in any neighborhood of the point.  If certain estimates hold, the metric will also be of class $C^{2,1}$ at the point.
\end{itemize}

The metric that Pogorelov uses for the first step is of the kind
\begin{equation}\label{radialmetric}
g_a=d\rho^2 + f_a(\rho)^2d\theta^2
\end{equation}
where
$$f_a(\rho) = \begin{cases}
\rho & \rho\le a/2\\
\rho + a(\rho-a)^3(\rho-a/2)^3& a/2<\rho<a
\end{cases} $$

The coordinates $\rho$ and $\theta$ for metric of the form \eqref{radialmetric} are the geodesic polar coordinates relative to the center of the disc.  Since this metric is rotationally symmetric, it is natural to ask first if there is a rotationally symmetric embedding of the metric into $\mathbb{R}^3$.  It is clear that not every rotationally symmetric Riemannian metric is so embeddable, even locally, as the example of any piece of the hyperbolic plane shows.  To give Pogorelov's result some plausibility, we discuss the general problem of embedding such a metric in a rotationally symmetric fashion.  It turns out that such an embedding does exist for Pogorelov's metric, but that it is only of class $C^{1,1}$.

Next, we shall prove Pogorelov's theorem, correcting some of the details of Pogorelov's paper, and including many others that he does not give.

\section{$C^{1,1}$ isometric embedding of Pogorelov's metric}

Let $D$ be a disc centered at the origin in $\mathbb{R}^2$, and define a metric in $D$ by
$$g = d\rho^2 + f(\rho)^2d\theta^2$$
where $f$ is some given function of class $C^2$.  We ask under what conditions $g$ can be realized as the metric of a surface of revolution in three-dimensional Euclidean space.  A surface of revolution is a surface represented in cylindrical coordinates by $r=g(z)$.  Equivalently, this is the rotation of the curve $r=g(z)$ about the $z$-axis.  The metric has the form
$$h = (1+g'(z)^2)dz^2 + g(z)^2d\theta^2.$$

We will here study only embeddings that preserve the angular parameter $\theta$.  We must therefore have
$$g(z) = f(\rho)$$
and
$$\left(\frac{d\rho}{dz}\right)^2=1+\left(\frac{dg}{dz}\right)^2.$$

Substituting the first equation into the second and applying the chain rule gives
$$\left(\frac{d\rho}{dz}\right)^2=1+\left(\frac{df}{d\rho}\frac{d\rho}{dz}\right)^2$$
so
$$\left(\frac{dz}{d\rho}\right)^2 = 1-\left(\frac{df}{d\rho}\right)^2.$$
So by integration of this equation, we obtain a $C^2$ isometric embedding in the region where $f'(\rho)<1$.

Now in the special case of Pogorelov,
$$f(\rho) = \begin{cases}
\rho & \rho\le a/2\\
\rho + a(\rho-a)^3(\rho-a/2)^3& a/2<\rho<a
\end{cases} $$
where $a$ is a parameter.  Here $f'(\rho)<1$ in the region $(a/2,3a/4)$, and so we obtain a $C^2$ embedding in that annular region.

Let
$$z(\rho) = \begin{cases}
0 & \text{if $\rho\le a/2$}\\
\int_{a/2}^\rho \sqrt{1-f'(t)^2}\,dt &\text{if $a/2<\rho<3a/4$.}
\end{cases}.$$
and
$$r(\rho) = f(\rho).$$
This defines the curve whose surface of revolution is an isometric embedding of the original surface into $\mathbb{R}^3$.  The curve $\rho\mapsto (z(\rho),f(\rho))$ is a regular $C^1$ curve.  However, it is not $C^2$ at $\rho=a/2$, since
$$\lim_{\rho\to a/2^-}z''(\rho) =0$$
but
$$\lim_{\rho\to a/2^+}z''(\rho) =\lim_{\rho\to a/2^+}\frac{-f''(\rho)f'(\rho)}{\sqrt{1-f'(\rho)^2}} = \frac{\sqrt{3}}{2}a^2.$$
Since the second derivative has only a jump discontinuity, the embedding is $C^{1,1}$.

\begin{figure}[htp]
\includegraphics[scale=0.6]{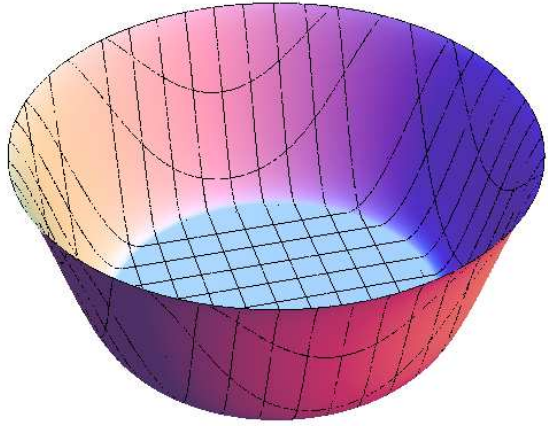}
\caption{$C^{1,1}$ isometric embedding of Pogorelov's metric as a surface of revolution}
\end{figure}

\section{Global nonembeddability}
Let $D_a$ be the disc $\rho<a$.  Equip $D_a$ with the $C^{2,1}$ metric
\begin{equation*}
g_a=d\rho^2 + f_a(\rho)^2d\theta^2
\end{equation*}
where
$$f_a(\rho) = \begin{cases}
\rho & \rho\le a/2\\
\rho + a(\rho-a)^3(\rho-a/2)^3& a/2<\rho<a
\end{cases}.$$
We here prove that there is no $C^2$ embedding of the Riemannian surface $(D_a,g_a)$.

\begin{lemma}
Let $S$ be a $C^2$ developable surface in $\mathbb{R}^3$ with non-zero mean curvature, let $k$ be the non-zero principal curvature at each point, and let $L$ be a generator of the ruling of $S$.  Then along $L$
$$k = \frac{A}{s+B}$$
where $A$ and $B$ are constants and $s$ is the arclength parameter of $L$.
\end{lemma}
\begin{proof}
Define the radius of normal curvature to be the reciprocal of the normal curvature.  The radius of normal curvature at $x\in L$ is the distance in the normal direction of $S$ through $x$ to the intersection of the normal plane through $L$ with the normal plane to an infinitely near generator $L'$.  This intersection is a line, so if we move $x$ a distance $s$ along $L$, the radius of normal curvature will change by a linear function, since the normal to a developable surface is parallel along a generator.
\end{proof}

Since the mean curvature cannot change sign on any affine segment of a developable surface:

\begin{lemma}\label{MeanCurvature}
Let $s>t>0$ and $\phi:D_s\to \mathbb{R}^3$ a $C^2$ (topological) embedding. Suppose that $\phi(D_s\setminus\overline{D_t})$ has non-negative mean curvature and $\phi(D_t)$ has zero Gauss curvature.  Then $\phi(D_s)$ has non-negative mean curvature throughout.
\end{lemma}

\begin{lemma}\label{convexfunction}
Let $z=z(x,y)$ be a $C^2$ convex function on a rectangle $[-c,c]\times [0,b]$.  Suppose that $z_x(x,0)=z_y(x,0)=0$ for all $x\in[-c,c]$ and that $m\le z_{yy}\le M$ throughout the rectangle.  Then $z_{xx}(x,b)\le (M-m)\frac{b^2}{c^2}$ at some point $x\in [-c,c]$.
\end{lemma}

\begin{proof}
Assume in addition, without loss of generality, that $z(x,0)=0$ for all $x\in[-c,c]$.  Consider the function $f(x) = z(x,b)$.  This is a convex function of one variable in $[-c,c]$.  By integrating $z_{yy}$ twice, we find
$$\frac{m}{2}b^2\le f(x) \le \frac{M}{2}b^2.$$
Assume for a contractiction that $f''(x)>(M-m)b^2/c^2$ for all $x$.  Thus $f(x)$ is a $C^2$ function on $[-c,c]$ satisfying
\begin{enumerate}
\item $f''(x)\ge 0$
\item $\frac{m}{2}b^2\le f(x) \le \frac{M}{2}b^2$
\item $f''(x)>(M-m)b^2/c^2$
\end{enumerate}
from which we hope to deduce a contradiction.

Suppose first that there is a minimum of $f(x)$ in $(-c,c)$, say $x^*$.  If $x^*>0$, then we replace the function $f(x)$ with $f(-x)$ (which still satisfies (1)-(3)) to ensure that $x^*\le 0$.  Now
$$\int_{x^*}^c\int_{x^*}^t f''(x)\,dx\,dt \ge \int_0^c\int_0^t f''(x)\, dx\, dt > \frac{(M-m)b^2}{2}$$
by (3).  But we may carry out the first integral to find
$$\int_{x^*}^c\int_{x^*}^t f''(x)\,dx\,dt =\int_{x^*}^c\left(f'(x) - \underbrace{f'(x^*)}_{=0}\right)\,dx = f(c)-f(x^*) \le \frac{M-m}{2}b^2$$
by (2), a contradiction.

On the other hand, if there is no minimum of $f(x)$ in $(-c,c)$, then by replacing $f(x)$ by $f(-x)$ again if necessary, we can assume that $f'(-c)\ge 0$.  Then
$$\int_{-c}^c\int_{-c}^t f''(x)\,dx\,dt > (M-m)b^2$$
by (3).  But
\begin{align*}
\int_{-c}^c\int_{-c}^t f''(x)\,dx\,dt &= \int_{-c}^c\left(f'(t)-f'(c)\right)\,dt = f(c)-f(-c) - 2cf'(c)\\
&\le f(c)-f(-c)\le \frac{(M-m)b^2}{2}
\end{align*}
by (2), a contradiction.
\end{proof}

\begin{definition}
Let $K\subset\mathbb{R}^2$ be a convex plane domain and $\phi:K\to\mathbb{R}^3$.  A line segment $L$ is called an {\em affine segment} if both endpoints are on the boundary of $K$ and $\phi|_L$ is affine.
\end{definition}

\begin{lemma}
Let $D\subset\mathbb{R}^2$ be a disc and $\phi:D\to\mathbb{R}^3$ be a $C^2$ isometric embedding.  There exist affine segments of arbitrarily small length.  Moreover, if an affine segment cuts the disc into two parts, then the smaller part contains affine segments of arbitrarily small length.
\end{lemma}

\begin{proof}
Assume for contradiction that there is a greatest lower bound $\epsilon>0$ to the length of affine segments.  The space of oriented line segments of length $\ge \epsilon$ with endpoints on $\partial D$ is a closed subset of $\partial D\times\partial D$, and is therefore compact.  Thus there is a minimizing sequence of affine segments $L_n\to L$ such that the length of $L_n$ tends to $\epsilon$.  Since the limit of affine maps is affine, $\phi|_L$ is also affine, and so $L$ is an affine segment.  Now $L$ cuts $D$ into two parts.  Let $D'$ be the smaller of the two.  Then $\phi|_{D'} : D'\to\mathbb{R}^3$ is an isometric embedding.  Any line segment $L'$ in $D'$ on which $\phi$ is affine must either intersect $L$ or have both points on the boundary of $D$.  The second case is impossible, in view of the minimality of $L$, and so that leaves only the first case to consider.  In that case the convex hull of $L$ and $L'$ inside $D'$ is a triangle on which $\phi$ is affine.  Some leg of this triangle is an affine segment shorter than $L$, which again contradicts its minimality.

For the second assertion, let $E$ be the smaller domain cut from $D$ by $L$.  First, note that there is at least one affine segment (with endpoints in $\partial D\cap \partial E$) in $E$ that is shorter than $L$, by the preceding argument.  The set of all such affine segments is nonempty.  Assume for contradiction that there is a greatest lower bound $\epsilon>0$ to the lengths of such segments.  Using the same argument as before, we get an affine segment of length $\epsilon$ whose endpoints are in $\partial D\cap \partial E$.  But then we can derive a contradiction in the same way.
\end{proof}

\subsection{Gauss curvature}
\begin{lemma}\label{gausscurvature}
For a surface metric in geodesic normal polar coordinates \eqref{radialmetric}, the Gauss curvature is $K=-f_a''(\rho)/f_a(\rho)$.
\end{lemma}

With the particular function $f_a(\rho)$ in our case, we find
$$K=-\frac{6a(a-2r)(a-r)(11a^2-30ar+20r^2)}{a(a-2r)^3(a-r)^3+8r}.$$
Expanding to second order in $r-a/2$, we have
$$K=\frac{3}{2}a^3(\rho-a/2) - 21a^2(\rho-a/2)^2 + O(\rho-a/2)^3.$$
For $\epsilon>0$ sufficiently small,
\begin{equation}\label{GaussCurvature}
K>\frac{3}{4}a^3(\rho-a/2)
\end{equation}
in the annular region $a/2<\rho<a/2+\epsilon$.

\begin{theorem}
For $a$ sufficiently small, there is no $C^2$ isometric embedding of the metric $g$ into $\mathbb{R}^3$.
\end{theorem}

\begin{proof}
Suppose for contradiction that there exists such an embedding $\phi:(D_a,g_a)\to \mathbb{R}^3$.  Let $D\subset D_a$ be the disc $0\le\rho<a/2+\epsilon$ where $\epsilon$ is small, and $D_0$ be the smaller disc $0\le\rho<a/2$.  Then $\phi(D)$ has non-negative Gauss curvature and strictly positive Gauss curvature in a neighborhood of the boundary.  By Lemma \ref{MeanCurvature}, $\phi(D)$ has non-negative mean curvature throughout.

Let $c>0$ be small relative to $a$ and $\epsilon$.  By choosing $c$ even smaller if necessary, we can assume that there is an affine segment $L$ of $D_0$ of length $2c$.  By applying a Euclidean transformation, we can assume that $L$ is the interval $[-c,c]$ of the $x$-axis and that the surface is tangent to the $xy$-plane along $L$.  A neighborhood of $L$ is realizable as a graph over the $xy$ plane.  By choosing $c$ smaller still, we can assume that a portion of the surface is realizable as a graph $z=z(x,y)$ over a box $[-c,c]\times [0,b]$ where $b=3c^2/a$.

Since both principle curvatures are non-negative, $z$ is convex (or concave; in that event, we perform the reflection $z\to -z$).  Now, along $[-c,c]$, $z_x=z_y=0$.  Let $M=\max_{[-c,c]\times [0,b]} z_{yy}$, $m=\min_{[-c,c]\times [0,b]} z_{yy}$.  Then by Lemma \ref{convexfunction}, there exists an $x\in [-c,c]$ such that $z_{xx}(x,b)\le (M-m)\frac{b^2}{c^2} =  (M-m)\frac{9c^2}{a^2}$.  The point $(x,b)$ will lie a geodesic distance $\ge c^2/a$ outside the circle $\rho=a/2$.  Indeed, the altitude of the triangle formed by the center of the circle $\rho=a/2$ and the points $(\pm c,0)$ in $D_a$ has length $\sqrt{a^2/4 - c^2}$, so the sagitta to the segment $[-c,c]\times\{b\}$ is
\begin{align*}
\frac{a}{2} -  \sqrt{\frac{a^2}{4} - c^2} &= \frac{a}{2}\left(1-\sqrt{1-\frac{4c^2}{a^2}}\right)\\
&= \frac{a}{2}\left(1-\left(1-\frac{2c^2}{a^2}\right)\right) + O(c^4) = \frac{c^2}{a} + O(c^4) \le \frac{2c^2}{a}
\end{align*}
since $c$ is small.  

But by the discussion following Lemma \eqref{gausscurvature}, the Gauss curvature at $(x,b)$ is greater than $\frac{3}{4}a^3\left(\frac{c^2}{a}\right) = 3a^2c^2/4$.  Indeed, we have shown that $K>\frac{3}{4}a^2(\rho-a/2)$.  On the segment ${b}\times [-c,c]$ of the $xy$ plane, the midpoint is the closest to the curve $\rho=a/2$, with distance $\ge c^2/a$ (by the preceding calculation).

Since $|z_{xx}|\le (M-m)9c^2/a^2$ at $P$, the maximum principal curvature at $P$ is greater than $\frac{a^4}{9(M-m)}$.  Now, $M-m$ depends on the $c$ and the original affine segment selected.  We choose a family of affine segments, one subtending the other, such that $c\to 0$.  Doing this, since $z$ is $C^2$, we have $M-m\to 0$ as well.  But this implies that the principal curvature blows up.
\end{proof}

\section{Local nonembeddability}
Let $D_n$ be the closed disc of radius $1/2(n+1)^2$ centered at $(1/n,0)$ for $n=1,2,\dots$.  These are disjoint discs whose radii tend to zero and whose centers tend to the point $(0,0)$.  In each disc, define a metric $h_n$ to be the metric $g_{1/2(n+1)^2}$ constructed in the preceding section, where $\rho$ is the distance to, and $\theta$ is the polar angle about, the point $(1/n,0)$ .  Then, let $h$ be the flat metric on the complement of $\bigcup_n D_n$, and the metric $h_n$ on $D_n$.

\begin{lemma}
$h\in C^{2,1}$
\end{lemma}

\begin{proof}
We estimate the norms of the coefficients on the derivatives of the metrics $h_n$ in the $xy$-coordinates.  Let $\delta$ be the Euclidean metric.  Then as $n\to\infty$, we estimate by computing $(f_a(\rho)^2/\rho^4)'$ and $(f_a(\rho)^2/\rho^4)''$ that
\begin{align*}
\|h_n-\delta\|_{\infty} &= O\left(\frac{1}{(n+1)^{20}}\right)\\
\|Dh_n\|_{\infty} &= O\left(\frac{1}{(n+1)^6}\right)\\
\|D^2h_n\|_{\infty} &= O\left(\frac{1}{(n+1)^4}\right)\\
[D^2h_n]_{Lip} &= O\left(\frac{1}{(n+1)^2}\right)
\end{align*}
where the Lipschitz constant is obtained from $(f_a(\rho)^2/\rho^2)'''$ for $\rho>a/2$.  The partial sums of $h_n-\delta$ form a Cauchy sequence in $C^{2,1}$, and therefore $h$ is $C^{2,1}$.
\end{proof}

\begin{theorem}
There is no $C^2$ isometric embedding the metric $h$ in any neighborhood of the origin.
\end{theorem}

\begin{proof}
Any such neighborhood must contain infinintely many discs $D_{1/2(n+1)^2}$.  Restricting the isometric embedding to any such disc gives a $C^2$ isometric embedding of the metric $g_{1/2(n+1)^2}$.  Choosing $n$ to be sufficiently large that the results of the preceding section hold gives a contradiction.
\end{proof}

\bibliography{pogorelov}{}
\bibliographystyle{plain}

\end{document}